\documentclass[11pt,oneside,leqno]{amsart}
\usepackage{amsxtra}
\usepackage{xy}
\usepackage{amsopn}
\usepackage{color}
\usepackage{amsmath,amsthm,amssymb}
\usepackage{amscd}
\usepackage{amsfonts}
\usepackage{latexsym}
\usepackage{verbatim}
\usepackage{pb-diagram}

\theoremstyle{plain}
\newtheorem{theorem}{Theorem}[section]
\newtheorem*{theorem*}{Theorem}
\newtheorem{definition}[theorem]{Definition}
\newtheorem{lemma}[theorem]{Lemma}
\newtheorem{prop}[theorem]{Proposition}
\newtheorem*{prop*}{Proposition}
\newtheorem{cor}[theorem]{Corollary}
\newtheorem*{cor*}{Corollary}
\newtheorem{rem}[theorem]{Remark}
\newtheorem{ex}[theorem]{Example}
\newtheorem*{mt*}{Main Theorem}
\newcommand\C{{\mathbb C}}
\newcommand\D{{\mathbb D}}

\newcommand\R{{\mathbb R}}
\newcommand\Z{{\mathbb Z}}
\newcommand\T{{\mathbb T}}

\newcommand{\Cpf}{$\mathcal{C}^\infty$-pure-and-full}

\newcommand{\Cf}{$\mathcal{C}^\infty$-full}
\newcommand{\Cp}{$\mathcal{C}^\infty$-pure}
\newcommand{\pf}{pure-and-full}
\newcommand{\Pf}{Pure-and-full}
\newcommand{\fu}{full}
\newcommand{\pu}{pure}

\newcommand{\correnti}{\mathcal{D}}
\newcommand{\p}{pure}

\begin{document}
\title[On non-pure forms on almost complex manifolds]{On non-pure forms on almost complex manifolds}
\author{Richard Hind, Costantino Medori, Adriano Tomassini}
\date{\today}
\address{Department of Mathematics \\
University of Notre Dame \\
Notre Dame, IN 46556}
\email{hind.1@nd.edu}
\address{Dipartimento di Matematica e Informatica\\
Universit\`{a} di Parma \\
Parco Area delle Scienze 53/A, 43124 \\
Parma, Italy}
\email{costantino.medori@unipr.it}
\email{adriano.tomassini@unipr.it}
\thanks{Partially supported by {\em Fondazione Bruno Kessler-CIRM (Trento)} and by GNSAGA
of INdAM}
\begin{abstract}
In \cite{li-zhang} T.-J. Li and W. Zhang defined an almost complex structure $J$ on a manifold $X$ to be {\em \Cpf}, if the second de Rham cohomology group can be decomposed as a direct sum of the subgroups whose elements are cohomology classes admitting $J$-invariant and $J$-anti-invariant representatives. It turns out (see T. Draghici, T.-J. Li and W. Zhang \cite{draghici-li-zhang}) that any almost complex structure on a $4$-dimensional compact manifold is \Cpf.  We study the $J$-invariant and $J$-anti-invariant cohomology subgroups on almost complex manifolds, possibly non compact. In particular, we prove an analytic continuation result for anti-invariant forms on almost complex manifolds.
\end{abstract}

\maketitle
%\tableofcontents
\section*{Introduction}
Let $(X,J,\omega)$ be a compact K\"ahler manifold of complex dimension $n$. Then the celebrated Hodge decomposition Theorem states the $(p,q)$-decomposition on forms transfers at the cohomology level, namely, $H^k(X;\C)\cong \bigoplus_{p+q=k}H^{p,q}_{\overline{\partial}}(X)$, where
$H^{p,q}_{\overline{\partial}}(X)$ denotes the $(p,q)$-Dolbeault group. In order to generalize the previous decomposition to arbitrary compact almost complex manifolds, in \cite{li-zhang} T.-J. Li and W. Zhang introduced the notion of \Cpf \ almost complex structures. An almost complex structure $J$ on a $2n$-dimensional manifold $X$ is said to be {\em \Cpf} if $H^{2} (X; \R) =H^{(1,1)}_J (X)_{\R} \oplus H^{(2,0),(0,2)}_J (X)_{\R}$, where $H^{(1,1)}_J (X)_{\R}$ and $H^{(2,0),(0,2)}_J (X)_{\R}$ are the subgroups of the $2$-nd de Rham cohomology group of $X$, whose elements are cohomology classes represented by real pure forms of type $(1,1)$ and $(2,0)+(0,2)$, respectively, namely, the $J$-{\em invariant}, $J$-{\em anti-invariant} cohomology subgroups. More in general, $J$ is said to be {\em \Cp} if $H^{(1,1)}_J (X)_{\R} \cap H^{(2,0),(0,2)}_J (X)_{\R}=\{0\}$ and {\em \Cf} if $H^{2} (X; \R) =H^{(1,1)}_J (X)_{\R} + H^{(2,0),(0,2)}_J (X)_{\R}$. Similar notions of \Cp\ and \Cf\ almost complex structures at stage $k$ can be defined. By taking the complex of $(p,q)$-currents one can define {\em pure} and {\em full} almost complex structures.\newline
In \cite{draghici-li-zhang}, T. Dr\u{a}ghici, T.-J. Li and W. Zhang showed that an almost complex structure $J$ on a $4$-dimensional compact manifold $X$ is
\Cpf, namely $H^{2} (X; \R) =H^{(1,1)}_J (X)_{\R} \oplus H^{(2,0),(0,2)}_J (X)_{\R}$. The last decomposition can be considered as a Hodge decomposition for $4$-dimensional almost complex compact manifolds. Furthermore, setting
$$
\mathcal{K}^t_J = \left\{\left[\omega\right]\in H^2(X;\R) \,\,\,\vert\,\,\, \omega \text{ is symplectic and }J\text{ is }\omega\text{-tamed}\right\}
$$
the {\em tamed symplectic cone} and
$$
\mathcal{K}^c_J =\left\{\left[\omega\right]\in H^2(X;\R) \,\,\,\vert\,\,\, \omega \text{ is symplectic and } J\text{ is }
\omega\text{-compatible}\right\}
$$
the {\em compatible symplectic cone}, if $J$ is a complex structure on a compact manifold $M$ which admits a K\"ahler metric, then
$$
\mathcal{K}^t_J \;=\; \mathcal{K}^c_J \;+\; \left(\left( H^{2,0}_{\overline{\partial}}\left(X\right)
\,\oplus\, H^{0,2}_{\overline{\partial}}\left(X\right) \right) \,\cap\, H^2\left(X;\R\right) \right)
$$
(see \cite{li-zhang}). For a non-integrable almost complex structure $J$, in \cite{li-zhang} it is proved that if $J$ is a \Cf\  almost complex structure which carries an almost K\"ahler metric, then
$$
\mathcal{K}^t_J \;=\; \mathcal{K}^c_J \,+\, H^{(2,0),(0,2)}_J(X)_\R.
$$
Hence the subgroup $H^{(2,0),(0,2)}_J(X)_\R\subset H^2(M;\R)$ can be thought as a generalization of real Dolbeault group $\left( H^{2,0}_{\overline{\partial}}\left(X\right)
\,\oplus\, H^{0,2}_{\overline{\partial}}\left(X\right) \right) \,\cap\, H^2\left(X;\R\right)$ to the non-integrable case.\newline
For other results on \Cpf  \ almost complex structures see e.g. \cite{draghici-li-zhang1}, \cite{angella-tomassini}, \cite{angella-tomassini-zhang}.

In dimension higher than four, there are examples of non \Cp \ and non \Cf\  almost complex structures on compact manifolds. Nevertheless, in \cite{draghici-li-zhang} and \cite{fino-tomassini}, it is proved that if $(X,\omega)$ is a compact symplectic manifold of dimension $2n$, then every $\omega$-compatible almost complex structure $J$ is \Cp.\smallskip

In this paper we are interested in studying the $J$-invariant and $J$-anti-invariant respectively cohomology subgroups on almost complex manifolds $X^{2n}$, possibly non compact. In particular we will search for pairs $(\omega,J)$ where $\omega$ is either a $J$-invariant or $J$-anti-invariant closed $2$-form. We are also interested in determining whether $\omega$ can also be taken to be symplectic. We note that symplectic forms always have compatible almost complex structures, which are invariant in our language. Therefore in the symplectic case we will focus on finding anti-invariant almost complex structures.

After some preliminaries in section \ref{preliminaries}, in section \ref{compact} we start with the observation that if $n$ is odd and $(\omega,J)$ is an anti-invariant pair then $\omega^n$ is identically zero, see Proposition \ref{non-anti-invariant}. In particular $\omega$ cannot be symplectic. However it is not true that lower powers of $\omega$ must vanish, even in cohomology. In Example \ref{nonnull-square} we provide an example of a $2$-cohomology class on a $6$-dimensional compact nilmanifold $N$ endowed with a left-invariant almost complex structure $J$, which belongs to $H^{(1,1)}_J (N)_{\R} \cap H^{(2,0),(0,2)}_J (N)_{\R}$ and whose square is non-zero.

In searching for anti-invariant forms with $\omega$ symplectic then we must assume that $n$ is even. We consider the case of open manifolds. Theorem \ref{anti-invariant-symplectically-trivial} says that on an open symplectic $4$ manifold $(X^4, \omega)$ there exists an anti-invariant $J$ if and only if the tangent bundle is symplectically trivial. However Proposition \ref{pontrjagin} shows that this is no longer true in higher dimensions. If we do not fix the symplectic form, Theorem \ref{opensymp} says that given any open manifold $X^{2n}$ with $n$ even and $TX$ trivial, and any cohomology class $\mathfrak{a} \in H^2(X, \R)$, there exists an anti-invariant pair $(\omega,J)$ with $[\omega] = \mathfrak{a}$.

In the remainder of the paper we are motivated by the problem of finding anti-invariant pairs $(\omega,J)$ in a given cohomology class without the restrictive symplectic assumption. As in Proposition \ref{non-anti-invariant} it is not hard to show, for example, that on a $6$-manifold such an $\omega$ must everywhere have rank $0$ or $4$. There are no other pointwise restrictions. Still, this seems restrictive, but we show in section \ref{duals} that for many codimension $2$ submanifolds $W$ we can find closed forms $\omega$ supported in arbitrarily small tubular neighborhoods such that $\omega$ has everywhere rank $0$ or $4$ and $[\omega]$ is Poincar\'{e} dual to $W$. Given this, it is natural to ask whether anti-invariant pairs can be constructed using cut and paste methods.

Section \ref{continuation} shows that this is impossible, indeed, despite the rank condition, the compactly supported forms constructed in section \ref{duals} have no anti-invariant almost complex structures. This is a consequence of the unique continuation theorem, see Theorem \ref{analytic-continuation-6}. This says that given an anti-invariant pair $(\omega,J)$, if $\omega$ vanishes on an open subset then it must vanish everywhere on a connected $X^{2n}$.
\smallskip

\noindent {\em Acknowledgements.} We would like to thank {\em Fondazione Bruno Kessler-CIRM (Trento)} for their support and very pleasant working environment. We also would like to thank Daniele Angella for useful comments.
\smallskip

\section{Preliminaries}\label{preliminaries}
\subsection{\Cpf\   almost complex structures}
Let $X$ be a $2n$-dimensional manifold (without boundary) and $J$ be a smooth almost complex structure on $X$.  Let
$(T^{1,0}X)^*$ and $(T^{0,1}X)^*$ be the bundle of complex $1$-forms on $X$ of type $(1,0)$ and $(0,1)$, respectively. Denote
by $\Lambda_\C^k(X)$ the bundle of $\mathcal{C}^\infty$ complex $k$-forms on $X$ and by the same symbol $\Lambda^k_\C(X)=\Gamma(M,\Lambda_\C^k(X))$ the space of sections of $\Lambda_\C^k(X)$. Then, as usual, setting
$\Lambda^{p,q}_J(X)=\Lambda^p\left((T^{1,0}X)^*\right)\wedge\Lambda^q\left((T^{0,1}X)^*\right)$, the space of complex $k$-forms decomposes as
$$
\Lambda^k_\C(X) =\bigoplus _{p+q=k}\Lambda^{p,q}_J(X)\,,
$$
where $\Lambda^{p,q}_J(X) =\Gamma(M,\Lambda^{p,q}_J(X))$.\newline
Accordingly to the above decomposition, the space $\Lambda^k_{\R} (X)$ of real smooth differential $k$-forms has a type decomposition:
$$
\Lambda^k_{\R} (X) = \bigoplus_{p + q = k} \Lambda^{(p,q),(q,p)}_J (X)_{\R},
$$
where
$$
\Lambda^{(p,q),(q,p)}_J (X)_{\R} = \left\{ \alpha \in \Lambda^{p,q}_J (X) \oplus \Lambda^{q,p}_J (X)
\, \vert \, \alpha = \overline \alpha \right\}\,.
$$
In particular, since $J$ acts as the identity on $\Lambda^{(1,1)}_J (X)_{\R}$ and as minus the identity on
$\Lambda^{(2,0),(0,2)}_J (X)_{\R}$, sometimes we will set
$$
\Lambda^+_J=\Lambda^{(1,1)}_J (X)_{\R}\,,\quad  \Lambda^-_J=\Lambda^{(2,0),(0,2)}_J (X)_{\R}
$$
and we will refer to them as the spaces of {\em $J$-invariant} and {\em $J$-anti-invariant} forms, respectively.

For a finite set $S$ of pairs of integers, let
$$
{\mathcal Z}^{S} _J = \displaystyle{\bigoplus_{(p, q) \in S}} {\mathcal Z}^{(p,q),(q,p)}_J, \quad
{\mathcal B}^{S} _J = \bigoplus_{(p, q) \in S} {\mathcal B}^{(p,q),(q,p)}_J,
$$
where
$$
{\mathcal Z}^{(p,q),(q,p)}_J=\{\alpha\in\Lambda^{(p,q),(q,p)}_J (X)_{\R} \,\,\,\vert\,\,\, d\alpha=0\}
$$
and
$$
{\mathcal B}^{(p,q),(q,p)}_J=\{\beta\in\Lambda^{(p,q),(q,p)}_J (X)_{\R}\,\,\,\vert\,\,\, \hbox{\rm there exists}\,\, \gamma\,\,\hbox{\rm such that}\,\, \beta=d\gamma\}\,.
$$
Denoting by $\mathcal{B}$ the space of $d$-exact forms, we have that
$$
\frac{{\mathcal Z}^{S} _J}{{\mathcal B}^{S} _J }=\frac{{\mathcal Z}^{S} _J}{\mathcal{B}\cap{\mathcal Z}^{S} _J }=\frac{{\mathcal Z}^{S} _J}{{\mathcal B} }\,.
$$
Hence, there is a natural inclusion
$$
\rho_S: {\mathcal Z}^{S} _J/{\mathcal B}^{S} _J \to {\mathcal Z}/{\mathcal B}\,.
$$
As in \cite{li-zhang} we will write
$\rho_S ({\mathcal Z}^{S} _J/{\mathcal B}^{S} _J)$ simply as ${\mathcal Z}^{S} _J/{\mathcal B}^{S} _J$
and we may define the cohomology subgroups
$$
H^S_J (X)_{\R}= \left\{ [\alpha ] \,\, \vert \, \,\alpha \in {\mathcal Z}^{S} _J \right\} =
\frac {{\mathcal Z}^{S} _J } {\mathcal B}.
$$
In particular, according to the above definition, there is a natural inclusion
$$
H^{(1,1)}_J (X)_{\R} + H^{(2,0),(0,2)}_J (X)_{\R} \subseteq
H^{2} (X, \R).
$$
We recall the following (see \cite[Definitions 2.2 and 2.3, Lemma 2.2]{li-zhang})
\begin{definition}
A smooth almost complex structure $J$ on $X$ is said to be
\begin{itemize}
\item {\em \Cp}\  if
$$
H^{(1,1)}_J (X)_{\R} \cap H^{(2,0),(0,2)}_J (X)_{\R}=\{0\};
$$
\item {\em \Cf}\  if
$$
H^2(X;\R)=H^{(1,1)}_J (X)_{\R} + H^{(2,0),(0,2)}_J (X)_{\R};
$$
\item {\em \Cpf} if
$$
H^{2} (X; \R) =H^{(1,1)}_J (X)_{\R} \oplus H^{(2,0),(0,2)}_J (X)_{\R}.
$$
\end{itemize}
\end{definition}
\subsection{\Pf\   almost complex structures} Let $(X,J)$ be a compact $2n$-dimensional almost complex manifold. Denote  by
$\correnti_k(X)$ the space of \emph{currents} of dimension $k$ (or degree $(2n-k)$), i.e., the topological dual of the space
$\Lambda^k(X)$ of $r$-forms on $X$ (see e.g  \cite{derham}, \cite{demailly-agbook}). The exterior differential $d$ on
$\Lambda^\bullet(X)$ induces a differential on $\correnti_\bullet (X)$, still denoted by $d$.
Then the de Rham homology $H_\bullet(X;\R)$ is the cohomology of the
differential complex $\left(\correnti_\bullet(X),\,d\right)$ and $H^k_{dR}(X;\R)\,\simeq\,H_{2n-k}(X;\R)$. We will denote by
$\mathcal{P}:H_{2n-k}(X;\R)\to H^k_{dR}(X;\R)$ this isomorphism. \newline
As in the case of complex forms, the almost complex structure $J$ induces a bi-grading on the space $\correnti_{k}^{\C}(X)$ of complex currents of dimension $k$. Hence, the spaces of currents $\correnti_{p,q}(X)$ of bidimension $(p,q)$ are defined.

Accordingly, let $H_{(2,0),(0,2)}^J(X)_\R$ (respectively, $H_{(1,1)}^J(X)_\R$) be the subspace of $H_2(X;\R)$
given by the homology classes represented by a real current of bidimension $(2,0)+(0,2)$ (respectively, $(1,1)$).
Then we have the following definition due to T.-J. Li and W. Zhang.
\begin{definition}[{\cite[Definition 2.5, Lemma 2.7]{li-zhang}}]
 An almost complex structure $J$ on $X$ is said to be:
\begin{itemize}
 \item \emph{\p} if
$$ H_{(2,0),(0,2)}^J(X)_\R\;\cap\; H_{(1,1)}^J(X)_\R \;=\; \left\{0\right\} \;; $$
 \item \emph{\fu} if
$$ H_{(2,0),(0,2)}^J(X)_\R\;+\; H_{(1,1)}^J(X)_\R \;=\; H_2(X;\R)\;; $$
 \item \emph{\pf} if it is both \pu\ and \fu, i.e. if the following decomposition holds:
$$ H_{(2,0),(0,2)}^J(X)_\R\,\oplus\, H_{(1,1)}^J(X)_\R \,=\, H_2(X;\R) \;.$$
\end{itemize}
\end{definition}

The notions of  ${\mathcal C}^{\infty}$-pure/full and pure/full almost complex structures can be given in a similar way for $k$-forms and $k$-currents, respectively.

The relations between being \Cpf\ and being \pf\ are summarized in the following.
\begin{prop}[{see \cite[Proposition 2.5]{li-zhang}, \cite[Theorem 2.1]{angella-tomassini}}]\label{thm:implicazioni}
 The following relations between \Cpf\ and \pf\ concepts hold:
$$
\begin{array}{ccc}
\text{\Cf\ at the }k\text{-th stage} & \Longrightarrow & \text{\pu\ at the }k\text{-th stage}\\[3pt]
\Downarrow &  & \Downarrow\\[3pt]
\text{\fu\ at the }\left(2n-k\right)\text{-th stage} & \Longrightarrow &  \text{\Cp\ at the }\left(2n-k\right)\text{-th stage}\,.
\end{array}
$$
\end{prop}
In the sequel we will use the following notation:
\begin{eqnarray*}
{\mathcal Z}^-_J&=&{\mathcal Z}^{(2,0),(0,2)}_J\\
{\mathcal Z}^+_J&=& {\mathcal Z}^{(1,1)}_J\\
{\mathcal B}^-_J&=& {\mathcal B}^{(2,0),(0,2)}_J\\
{\mathcal B}^+_J&=& {\mathcal B}^{(1,1)}_J\\
H^-_J(X)&=& H^{(2,0),(0,2)}_J (X)_{\R}\\
H^+_J(X)&=& H^{(1,1)}_J (X)_{\R}\\
H_-^J(X)&=& H_{(2,0),(0,2)}^J (X)_{\R}\\
H_+^J(X)&=& H_{(1,1)}^J (X)_{\R}\,.
\end{eqnarray*}
\begin{rem}
Let $SU(2)\times\T^3$ be the product of the special unitary group with the torus $\T^3=\R^3\slash\Z^3$.
Let $\{e^1,\ldots ,e^6\}$ be the global parallelism on $SU(2)\times\T^3$ satisfying the Maurer-Cartan equations
$$
\left\{
\begin{array}{l}
de^1=e^2\wedge e^3,  \\[5pt]
de^2=-e^1\wedge e^3,  \\[5pt]
de^3=e^1\wedge e^2,\\[5pt]
de^4=de^5=de^6=0.
\end{array}
\right.
$$
Define the almost complex structure $J$ on $SU(2)\times\T^3$ by giving the following complex $(1,0)$-forms
$$
\varphi^1= e^1+ie^4,\quad\varphi^2= e^2+ie^5,\quad\varphi^3= e^3+ie^6;
$$
then (see \cite[example 5.4]{angella-tomassini-1}) we have that $J$ is a \Cf\  non \Cp\  almost complex structure. Hence, accordingly to Proposition \ref{thm:implicazioni}, $J$ is a pure non full almost complex structure on $SU(2)\times\T^3$.

\end{rem}

\section{Compact manifolds} \label{compact}
\subsection{$J$-Invariant cohomology classes}
For general $2$-cohomology classes on arbitrary manifolds, we can show
\begin{prop}\label{non-anti-invariant}
Let $X$ be a $2n$-dimensional manifold. Let $J$ be an almost complex structure on $X$ and $[\alpha]\in H^2(X;\R)$ such that $[\alpha]^n\neq 0$. Then
\begin{enumerate}
\item[i)] $[\alpha]\notin H^-_J(X)_\R\cap H^{+}_J(X)_\R$.\newline
\item[ii)] Moreover, if $n$ odd, then $[\alpha]\notin H^{-}_J(X)$.
\end{enumerate}
\end{prop}
\begin{proof}
i) By contradiction. Let $[\alpha]\in H^{-}_J(X)\cap H^{+}_J(X)$; then $[\alpha]=[\alpha_1]=[\alpha_2]$, where $\alpha_1$ is $J$-anti-invariant and $\alpha_2$ is $J$-invariant.
Then $\alpha_1\wedge\alpha_2^{n-1}$ is a $J$-anti-invariant $2n$-form which does not vanish identically. Therefore, $J(\alpha_1\wedge\alpha_2^{n-1})=-\alpha_1\wedge\alpha_2^{n-1}$. Take
$\{e_1,\ldots, e_n,Je_1,\ldots ,Je_n\}$ a basis of $T_pX$. Then
$$
\begin{array}{lll}
J(\alpha_1\wedge\alpha_2^{n-1})(e_1,\ldots,e_n,Je_1,\ldots,Je_n)&=&\alpha_1\wedge\alpha_2^{n-1}(Je_1,\ldots , Je_n, -e_1,\ldots ,-e_n)\\
&=&\alpha_1\wedge\alpha_2^{n-1}(e_1,\ldots ,e_n,Je_1,\ldots ,Je_n)\,,
\end{array}
$$
i.e., $J(\alpha_1\wedge\alpha_2^{n-1}) = \alpha_1\wedge\alpha_2^{n-1}$. This is absurd.

ii) We observe that $\alpha^n\neq 0$ at some point (otherwise,
$[\alpha]^n = 0$). Assume by contradiction that $[\alpha]\in H^{-}_J(X)$. Then there exists
$\varphi\in \Lambda^{-}_J(X)$ such that $[\alpha]=[\varphi]$; furthermore, there exists a point $p\in X$ such that $\varphi^n(p)\neq 0$. Since
$\varphi\in\Lambda^{-}_J(X)$ and $n$ is odd, we have that $J\varphi^n=-\varphi^n$.
Arguing as in the proof of i), we obtain that $\alpha^n=0$. This is absurd.
\end{proof}
As a direct consequence of the last Proposition, in the compact case we have the following
\begin{cor}
Let $(X,\omega)$ be a $2n$-dimensional compact symplectic manifold, with $n$ odd. Let $J$ be an almost complex structure on $M$. Then $[\omega]\notin
H^{-}_J(X)$.
\end{cor}
The following provides an example of a $2$-cohomology class $\mathfrak{a}$ on a $6$-dimensional compact nilmanifold $N$ endowed with a left-invariant almost complex structure $J$, which belongs to
$H^{+}_J (N) \cap H^{-}_J (N)$ and whose square is non-zero; nevertheless, $\mathfrak{a}^3=0$, according to Proposition \ref{non-anti-invariant}.
\subsection{Example}\label{nonnull-square}  Let $\mathfrak{g}$ be the $6$-dimensional nilpotent Lie algebra whose dual vector space admits a basis $\{e^1,\ldots ,e^6\}$ satisfying the following Maurer-Cartan equations:
$$
de^1=0\,,\qquad de^2=0\,,\qquad de^3=e^{12}\,,\qquad de^4=e^{13}\,,\qquad de^5=e^{14}\,,\qquad de^6=e^{23}\,,
$$
where $e^{ij}=e^i\wedge e^j$ and so on.
Let $G$ be the connected and simply-connected Lie group having $\mathfrak{g}$ as Lie algebra. Then, $G$ admits a uniform discrete subgroup $\Gamma$ and $N=\Gamma\backslash G$ is a $6$-dimensional nilmanifold. Let $J$ be the almost complex structure on $N$ defined by the following global $(1,0)$-forms
$$
\psi^1= e^1+ie^2\,,\quad \psi^2= e^4+ie^6\,,\quad \psi^3= e^3+ie^5\,.
$$
By Nomizu Theorem, it follows that $H^\bullet(\mathfrak{g})\simeq H^\bullet (N;\R)$ and consequently, $b_1(N)=2$, $b_2(N)=4$ and $b_3(N)=6$. By a straightforward computation it follows that
$0\neq [e^{26}]\in H^2(N;\R)$ and $0\neq [e^{15}]\in H^2(N;\R)$. Furthermore,
$$
[e^{26}]=[e^{26}-e^{14}]=[-\Re\mathfrak{e}\,(\psi^1\wedge\psi^2)]\,,\qquad [e^{26}]=[e^{26}+e^{14}]=[\Re\mathfrak{e}\,(\psi^1\wedge\overline{\psi}^2)]
$$
and
$$
[e^{15}]=[e^{15}+e^{23}]=[\Im\mathfrak{m}\,(\psi^1\wedge\psi^3)]\,,\qquad [e^{15}]=[e^{15}-e^{23}]=[\Im\mathfrak{m}\,(\overline{\psi}^1\wedge\psi^3)]
$$
Therefore, $0\neq [e^{26}]\in H^+_J(N)\cap H^-_J(N)$ and $0\neq[e^{15}]\in H^+_J(N)\cap H^-_J(N)$. Set $\mathfrak{a}=[e^{26}-e^{15}]\in H^+_J(N)\cap H^-_J(N)$ and
$\mathfrak{a}^2 =2[e^{1256}]\neq 0$.

\section{Non compact manifolds}
In this section we consider open manifolds (without boundary), namely manifolds without compact connected components. Starting with a symplectic $4$-manifold, we have
\begin{theorem}\label{anti-invariant-symplectically-trivial}
Let $(X,\omega)$ be an $4$-dimensional open symplectic manifold. Then there exists an almost complex structure $J$ on $X$ such that $[\omega]\in H^-_J(X)$ if and only if the tangent bundle $TX$ is symplectically trivial.
\end{theorem}
\begin{proof} $(\Leftarrow)$ Assume that $\{V_1,V_2,W_1,W_2\}$ is a global symplectic frame on $X$, namely
$$\omega(V_i,W_j)=\delta_{ij}\,,\quad\omega(V_i,V_j)=
\omega(W_i,W_j)=0\,,\quad i,j=1,\ldots, 4\,.
$$
Define
$$
JV_1=V_2\,,\quad JW_1=-W_2\,,\quad \omega =V^1\wedge W^1 + V^2\wedge W^2 \,,
$$
where $\{V^1,V^2,W^1,W^2\}$ is the dual frame of $\{V_1,V_2,W_1,W_2\}$. Then it is immediate to check that $J\omega=-\omega$, that is $[\omega]\in H^-_J(X)$.
\newline
$(\Rightarrow)$ Assume that there exists an almost complex structure $J$ such that $[\omega]\in H^-_J(X)$. As $X$ is open we can choose a never vanishing vector field
$V_1$ on $X$ and let $K$ be an $\omega$-calibrated almost complex structure on $X$. Since $\omega$ is anti-invariant, we have that
$$
\omega (V_1,JV_1)=0\,,\quad \omega (KV_1,JKV_1)=0\,;
$$
furthermore, $\omega(V_1,KV_1)>0$, $K$ being $\omega$-calibrated, and, consequently, $\omega(JV_1,JKV_1)<0$.
Then, by using the standard symplectic Gram-Schmidt process, we can extract a global symplectic basis from $\{V_1,JV_1,KV_1,JKV_1\}$.
\end{proof}
\begin{rem}\label{remark2}
We observe that the proof of $(\Leftarrow)$ holds for every $2n$-dimensional symplectic manifold with symplectically trivial tangent bundle, provided that $n$ is even.
\end{rem}
The next result shows that Theorem \ref{anti-invariant-symplectically-trivial} holds only in the four dimensional case. Namely,
\begin{prop}\label{pontrjagin}
Let $(X,J,g)$ be a compact almost Hermitian manifold of real dimension $4$ such that the Pontrjagin class $p_1(TX)\neq 0$ and $h^-_J(X)>0$.
Then the open almost complex $(4+4k)$-dimensional manifold $X\times \C^{2k}$ has a symplectic anti-invariant class such that $T(X\times \C^{4k})$ is nontrivial.
\end{prop}
\begin{proof}
By assumption, there exists $0\neq[\alpha]\in\mathcal{Z}^-_J$.
Let $(z_1,\ldots ,z_k,w_1,\ldots,w_k)$ be coordinates on $\C^{2k}$ and set $\omega=\sum_{j=1}^k dz_j\wedge dw_j$. Then $0\neq[\alpha + \omega]\in H^-_{\hat{J}}(X\times\C^{2k})$, where $\hat{J}$ is the product almost complex structure on $X\times\C^{2k}$. \newline
We have: $p_1(T(X\times \C^{2k})\big\vert_{X\times\{0\}})\in H^4(X)$ and $p_1(TX\oplus\C^{2k})=p_1(TX)\neq 0$ by assumption.
\end{proof}
\begin{ex} As an application of Proposition \ref{pontrjagin}, we take the {\em Fermat quartic}
$X=\{[z_0:z_1:z_2:z_3]\in\mathbb{P}^3(\C)\,\,\,\vert\,\,\, z_0^4+z_1^4+z_2^4+z_3^4=0\}$. It turns
out that the signature $\tau(X)=-16$ (see e.g. \cite[p.156]{joyce}). Then, according to the Hirzebruch signature Theorem, $p_1(TX)=3\tau(X)=-48$.
Furthermore, it is well known that $X$ has a nowhere vanishing holomorphic $(2,0)$-form $\Omega$. Hence, $[\Omega +\omega]\in H^2(X\times\C^{2k})$
is anti-invariant and $T(X\times\C^{2k})$ is non-trivial. More in general,  by \cite[Prop.3.21, Rem.3.22]{draghici-li-zhang1}, it follows
that on any K\"ahler surface with holomorphically trivial canonical bundle there exist (non-integrable) almost complex structures $J_{f,l,s}$,
depending on certain functions $f,l,s$, such that the dimension of $H^-_{J_{f,l,s}}(X)$ is $1$ or $2$. Therefore, we obtain
open almost complex $(4+4k)$-dimensional manifolds with an anti-invariant class and non-trivial tangent bundle.
\end{ex}
\begin{theorem} \label{opensymp}
Let $X$ be an open $2n$-dimensional manifold with trivial tangent bundle and $n$ even. Then for any cohomology class $\mathfrak{a}\in H^2(X)$ there exists an almost complex structure $J$ and an anti-invariant symplectic form $\omega$ on $X$ such that $\mathfrak{a}=[\omega]$.
\end{theorem}
\begin{proof}
Let $\{V_1,\ldots ,V_n,W_1,\ldots ,W_n\}$ be a global frame on $X$, $J_0$ be the almost complex structure on $X$ defined as $J_0 V_i=W_i,\, i=1,\ldots ,n$, and $g$ be the Hermitian metric defined by requiring that $\{V_1,\ldots ,V_n,W_1,\ldots ,W_n\}$ is an orthonormal frame. Then the fundamental form $\omega_0$ of $g$ is an almost symplectic form with symplectically trivial tangent bundle. Denote by
\begin{eqnarray*}
\mathbb{S}^{\mathfrak{a}}_{{\rm symp}}&=&\{\hbox{\rm symplectic forms $\omega$ on {\em X} such that } \mathfrak{a}=[\omega]\}\\
\mathbb{S}_{{\rm symp}}&=&\{\hbox{\rm symplectic forms on {\em X}} \}\\
\mathcal{S}_{{\rm symp}}&=&\{\hbox{\rm almost symplectic forms on {\em X}} \}\,.
\end{eqnarray*}
Then $\mathbb{S}^{\mathfrak{a}}_{{\rm symp}}\subset \mathbb{S}_{{\rm symp}}\subset \mathcal{S}_{{\rm symp}}$ and according to the $h$-principle  (see \cite{gromov1} or \cite[10.2.2]{eliashberg-mishashev}) the inclusion $\mathbb{S}^{\mathfrak{a}}_{{\rm symp}}\hookrightarrow \mathcal{S}_{{\rm symp}}$ is an homotopy equivalence. In other words there exists a family $\omega_t$ of almost symplectic forms on $X$ with $\omega_1$ symplectic and $\mathfrak{a}=[\omega_1]$.

Denote by $E\to X\times [0,1]$ the bundle over $X\times [0,1]$ whose fibre is given by
$$
E_x=\{\hbox{\rm ordered symplectic basis of $T_xX$ with respect to $\omega_t$}\}.
$$
Since $(X,\omega_0)$ is symplectically trivial, then there exists a section of $E$ over $X\times\{0\}$. By the homotopy lifting property there is a section of $E$ over $X\times [0,1]$. In particular $(X,\omega_1)$ is symplectically trivial. Hence, in view of Remark \ref{remark2}, it follows that there exists an almost complex structure $J$ such that
$\omega_1$ is a $J$-anti-invariant form.
\end{proof}
Concerning invariant cohomology classes, we have
\begin{prop}\label{invariant-cohomology-symplectic}
Let $(X,J)$ be an open almost complex manifold. Then given any non-zero class $\mathfrak{a}\in H^{2}(X;\R)$, there exists an almost complex structure $\hat{J}$ on $X$ which is homotopic to $J$ and such that
$\mathfrak{a}\in H^+_{\hat{J}}(X)$ and $\mathfrak{a}\notin H^-_{\hat{J}}(X)$.
\end{prop}
\begin{proof}
Let $0\neq \mathfrak{a}\in H^{2}(X;\R)$; then by the $h$-principle as above, see the remark in \cite{eliashberg-mishashev} following Theorem $10.2.2$, there exists a symplectic form $\omega$ and a homotopy equivalence of almost complex structures $\{\mathcal{J}_t\}$, for $t\in [0,1]$ such that
\begin{itemize}
\item $J\in \{\mathcal{J}_t\}$;
\item $\mathfrak{a}=[\omega]$ and $\mathcal{J}_1$ is compatible with $\omega$.
\end{itemize}
Hence, $\mathfrak{a}\in H^+_{\mathcal{J}_1}(X)$. We claim that $\mathfrak{a}\notin H^-_{\mathcal{J}_1}(X)$. On the contrary, $\mathfrak{a}\in H^+_{\mathcal{J}_1}(X)\cap H^-_{\mathcal{J}_1}(X)$; then $\mathfrak{a}=[\omega]=[\alpha_-]$, for suitable, $\alpha_-\in\mathcal{Z}^-$. Hence, we have
$$
0=\alpha_-\wedge\omega^{n-1}\,;
$$
therefore, $\alpha_-$ vanishes identically, and, consequently, $\mathfrak{a}=0$. This is absurd.
\end{proof}

\section{Analytic continuation of anti-invariant forms} \label{continuation}

In this section we prove an analytic continuation result for anti-invariant closed forms on almost complex manifolds.

\begin{theorem}\label{analytic-continuation-6}
Let $(X,J)$ be a connected $2n$-dimensional almost complex manifold. Let $\alpha$ be a closed $J$-anti-invariant form on $X$ such that $\alpha=0$ on an open set $\mathcal{U}\subset X$. Then $\alpha =0$ on $X$.
\end{theorem}

The proof in the $4$-dimensional case is
a direct consequence of Hodge theory.

\begin{proof} ($4$ dimensional case)
We may assume that $X=\R^4$. Take a $J$-Hermitian metric $g$ on $\R^4$; let $*_{g}$ be the Hodge operator and $\omega (\cdot ,\cdot)=g(\cdot ,J\cdot)$
be the fundamental form of $g$. Denote as usual by $\Lambda^+_{g}$ the space of self-dual forms on $(\R^4,g)$. Then,
$$
\Lambda^+_{g} = \R\,\omega \oplus \Lambda^-_J\,.
$$
Hence, if $\alpha\in\mathcal{Z}^-$, then $d*\alpha = d\alpha=0$. Therefore, $\Delta\alpha=0$ and by assumption $\alpha =0$ on $\mathcal{U}$.
In view of the Unique Continuation Theorem by Aronszajn (see \cite{aronszajn} and \cite{kazdan} or \cite{deturk}) we obtain that $\alpha =0$ on $\R^4$.
\end{proof}

The general case is more subtle but eventually reduces again to Aronszajn's Theorem.

\begin{proof} (general case)

Let $\D$ be a $J$-holomorphic disk in $X$. We can identify a small neighborhood of $\D \subset X$ with the total space of a complex vector bundle $\D \times \C^{n-1} \to \D$ such that $J|_{\D \times \C^{n-1}}$ coincides with the standard product structure along $\D \times \{0\}$. Let $t :(z,w) \mapsto (z,tw)$ denote the scaling in the fibers.

We linearize $J$ on $\D \times \C^{n-1}$ by replacing $J$ by $$J_0 = \lim_{t \to 0} (\frac{1}{t})^* \circ J \circ t_*.$$ Then $J_0$ is an almost complex structure on $\D \times \C^{n-1}$ invariant under scaling and restricting to the standard complex structure on the fibres $\{x\} \times \C^{n-1}$.

We can also replace $\alpha$ by $\alpha_0$ defined by $$\alpha_0(U,V) = \lim_{t \to 0} (\frac{1}{t}) \alpha(tU,tV).$$

{\bf Claim.} $\alpha_0$ is a closed, smooth $J_0$-anti-invariant form on $\D \times \C^{n-1}$. It vanishes on the fibers $\{x\} \times \C^{n-1}$.

To justify the claim we compute the almost complex structures and anti-invariant forms explicitly in local coordinates.

Let $(x_1, \dots ,x_{2n})$ be local coordinates in a neighborhood of a point $0 \in \D \times \{0\}$ such that the scaling is given by $t(x_1, \dots ,x_{2n})=(x_1,x_2,tx_3, \dots ,tx_{2n})$. Denote the coordinate vector fields by $e_i = \frac{\partial}{\partial x_i}$. Then $t_* e_i=e_i$ for $i=1,2$ and $t_* e_i = te_i$ for $i>2$. We can write
$$
\alpha = \sum_{i<j} a_{ij} dx_i \wedge dx_j
$$
and
$$
Je_i = \sum_j b_{ij}e_j
$$
where the $a_{ij}$ and $b_{ij}$ are functions of $(x_1, \dots ,x_{2n})$. Now, as the zero section is holomorphic and the fibers of the normal bundle are holomorphic along the zero
section we may assume that $b_{1j}(0) = \delta_{2j}, b_{2j}(0)= - \delta_{1j}$ and $b_{i1}(0)=b_{i2}(0)=0$ for all $i>2$. Furthermore, we will also suppose that coordinates are chosen such
that the matrix $(b_{ij})_{i,j >2}$ is the standard complex structure at all points $(x_1,x_2,0 \dots ,0)$. Next, as anti-holomorphic forms necessarily vanish on complex planes we have that $a_{12}(0)=0$.

We compute at a fixed point $x$ in the fiber over $0$.
$$
J_t(e_1(x))= (\frac{1}{t})^* J(t_*e_1(x)) = (\frac{1}{t})^*J(e_1(tx))
$$
$$
=(\frac{1}{t})^* \sum_j b_{1j}(tx) e_j = b_{11}(tx)e_1 + b_{12}(tx)e_2 + \sum_{j>2} b_{1j}(tx)\frac{1}{t}e_j
$$
$$
\to e_2 + \sum_{j>2} (\nabla b_{1j}(0) \bullet x)e_j
$$
as $t \to 0$.

Similarly $J_t(e_2(x))$ is a smooth vector field converging to
$$
-e_1 + \sum_{j>2} (\nabla b_{2j}(0) \bullet x)e_j
$$
as $t \to 0$.

For $i>2$ we have
$$
J_t(e_i(x))= (\frac{1}{t})^* J(t_*e_i(x)) = (\frac{1}{t})^*J(tx)(te_i)
$$
$$
=(\frac{1}{t})^* \sum_j b_{ij}(tx) te_j = b_{i1}(tx)te_1 + b_{i2}(tx)te_2 + \sum_{j>2} b_{ij}(tx)e_j
$$
$$
\to \sum_{j>2} b_{ij}(0) e_j
$$
as $t \to 0$. Hence in these coordinates the fibers converge to complex vector spaces as $t \to 0$.

For the forms, we have
$$
\alpha_t(x)(e_1,e_2)= \frac{1}{t} \alpha(tx)(e_1,e_2) = \frac{1}{t} a_{12}(tx)
$$
$$
\to \nabla a_{12}(0) \bullet x
$$
as $t \to 0$.

For $i,j>2$ we have
$$
\alpha_t(x)(e_i,e_j)= \frac{1}{t} \alpha(tx)(te_i,te_j) = ta_{ij}(tx) \to 0
$$
as $t \to 0$.

Finally for $i>2$ we have
$$
\alpha_t(x)(e_1,e_i)= \frac{1}{t} \alpha(tx)(e_1,te_i) = a_{1i}(tx) \to a_{1i}(0)
$$ as $t \to 0$ and
$$
\alpha_t(x)(e_2,e_i)= \frac{1}{t} \alpha(tx)(e_2,te_i) = a_{2i}(tx) \to a_{2i}(0)
$$
as $t \to 0$. As the convergence here is smooth, $\alpha_0$ is closed and $J_0$-anti-invariant as claimed.
 
Now, by the computations above, to understand $\alpha_0$ we are only interested in the values of $a_{1i}$ and $a_{2i}$ for $i>2$ along the zero-section, and hence will think of these only as functions of $x_1$ and $x_2$. The function $a_{12}$ on the other hand we must still consider as a function of $x_1, \dots ,x_{2n}$, although to simplify notation it is enough to consider the case when it is linear when restricted to the fibres.

There are relations amongst the functions above which we now detail. First of all, as a limit of closed forms, $\alpha_0$ is closed. Calculating its $dx_1 \wedge dx_2 \wedge dx_i$ component for $i>2$ we get
\begin{equation} \label{new1}
\frac{\partial a_{12}}{\partial x_i} -  \frac{\partial a_{1i}}{\partial x_2} + \frac{\partial a_{2i}}{\partial x_1} =0.
\end{equation}

Now we consider the anti-invariance of $\alpha_0$ at points along the zero-section and obtain

\begin{equation} \label{new2}
a_{1, (2i-1)}=-a_{2,2i}, \qquad a_{1, 2i}=a_{2, (2i-1)}
\end{equation}
for $i>1$. For instance, to prove the first of these, we use the formula
$
\alpha_0(0)(e_1,e_{2i-1}) = -\alpha_0(0)(e_2,e_{2i}).
$

Next, using our formulas for $\alpha_0$, we compute
$$
0= \alpha_0(e_1,J_0 e_1) = \alpha_0(e_1,e_2 + \sum_{j,k>2} \frac{\partial b_{1j}}{\partial x_k} x_k e_j)
$$
$$
= \sum_{k>2} \frac{\partial a_{12}}{\partial x_k} x_k + \sum_{j,k>2} a_{1j} \frac{\partial b_{1j}}{\partial x_k} x_k.
$$

As this holds for all choices of $x_k$ we obtain

\begin{equation} \label{new3}
\frac{\partial a_{12}}{\partial x_k} = -\sum_{j>2} a_{1j} \frac{\partial b_{1j}}{\partial x_k}.
\end{equation}

Applying equations \eqref{new2}, \eqref{new1} and then \eqref{new3} we obtain

\begin{equation} \label{new4}
\frac{\partial a_{1,(2i-1)}}{\partial x_2} - \frac{\partial a_{1,2i}}{\partial x_1} = \frac{\partial a_{1,(2i-1)}}{\partial x_2} - \frac{\partial a_{2,(2i-1)}}{\partial x_1} \\
= \frac{\partial a_{12}}{\partial x_{2i-1}} = -\sum_{j>2} a_{1j} \frac{\partial b_{1j}}{\partial x_{2i-1}}
\end{equation}

for all $i>1$, and similarly

\begin{equation} \label{new5}
\frac{\partial a_{1,2i}}{\partial x_2} + \frac{\partial a_{1,(2i-1)}}{\partial x_1} =- \sum_{j>2} a_{1j} \frac{\partial b_{1j}}{\partial x_{2i}}.
\end{equation}

Differentiating $\eqref{new4}$ with respect to $x_2$ gives

\begin{equation} \label{new6}
\frac{\partial^2 a_{1,(2i-1)}}{\partial x_2^2} - \frac{\partial^2 a_{1,2i}}{\partial x_1 \partial x_2} = -\sum_{j>2} \left( \frac{\partial a_{1j}}{\partial x_2} \frac{\partial b_{1j}}{\partial x_{2i-1}} + a_{1j} \frac{\partial^2 b_{1j}}{\partial x_2 \partial x_{2i-1}} \right)
\end{equation}

and differentiating $\eqref{new5}$ with respect to $x_1$ gives

\begin{equation} \label{new7}
\frac{\partial^2 a_{1,2i}}{\partial x_1 \partial x_2} + \frac{\partial^2 a_{1,(2i-1)}}{\partial x_1^2} = -\sum_{j>2} \left( \frac{\partial a_{1j}}{\partial x_1} \frac{\partial b_{1j}}{\partial x_{2i}} +
a_{1j} \frac{\partial^2 b_{1j}}{\partial x_1 \partial x_{2i}} \right).
\end{equation}

Adding equations \eqref{new6} and \eqref{new7} gives a formula for $\Delta a_{1,(2i-1)}$ in terms of the $a_{1j}$ for $j>2$ and their first derivatives. In fact,
if we think of the $a_{1j}$, $a_{2k}$ as giving a function $a:\D \to \R^{4(n-1)}$, $(x_1,x_2) \mapsto (a_{13}, \dots ,a_{1,2n},a_{23}, \dots ,a_{2,2n})$, then the above calculations give a bound on the Laplacian $\Delta a$ in terms of bounds on the function and its first derivatives $\frac{\partial a}{\partial x_1}$ and $\frac{\partial a}{\partial x_2}$. Thus by Aronszajn's Theorem the function $a$ satisfies a unique continuation theorem, in particular if $a$ is identically zero near a point then it is zero everywhere. Hence if the function
$a$ vanishes near a point in $\D$ then it vanishes everywhere on $\D$. As $a_{12}$ is identically zero along the zero-section $\D$ (as anti-invariant forms vanish on holomorphic curves) we conclude that if $\alpha =0$
near a point of $\D$ then it is zero at all points of $\D$.

To conclude the proof we need the following lemma.

\begin{lemma} \label{GW} Fix a Riemannian metric on $X$ defining a distance function $d$ and let $K \subset X$ be compact. There exists an $\epsilon>0$ such that for any $x,y \in K$ with $d(x,y) < \epsilon$
there exists a holomorphic disk $\D$ in $X$ passing through $x$ and $y$.
\end{lemma}

This follows from \cite{sikorav}, Theorem $3.1.1$ (i).

The lemma, together with the unique continuation theorem established along disks, implies that the set of points where $\alpha$ vanishes identically in
a neighborhood is open. As it is clearly closed we can conclude the proof of Theorem \ref{analytic-continuation-6}.
\end{proof}

\section{Closed forms of rank $0$ and $4$} \label{duals}

In this section we construct closed forms which
have everywhere rank $0$ or $4$, have compact support near a submanifold $W^{2n-2}$ of $X^{2n}$, and are Poincar\'{e} dual to $[W] \in H_2(X)$. As explained in the introduction, although there are no pointwise obstructions, the unique continuation theorem, Theorem \ref{analytic-continuation-6}, shows that there do not exist corresponding anti-invariant almost complex
structures. This is our main proposition.

\begin{prop}\label{non-pure}
Let $W\subset X^{2n}$ be a $(2n-2)$-dimensional compact submanifold with trivial normal bundle $\nu(W)$ and $\sigma$ be a $1$-form on $W$ such that
$d\sigma$ never vanishes and $d(\sigma\wedge d\sigma)=0$. Then there is a compactly supported $2$-form $\omega$ on the total space of $\nu(W)$ which everywhere
has rank either $4$ or $0$ and is cohomologous to the Thom class $\tau$ of $\nu(W)$.
\end{prop}
\begin{proof}
Let $\tau =r(x_1,y_1)dx_1\wedge dy_1$, be the local expression of the Thom class of $\nu(W)$, where $r(x_1,y_1)$ is a bump function such that $r=1$ on a neighborhood of
$W$ in $TW$. Define
$$
\omega =\tau + d(r\sigma)\,.
$$
Then,
$$
\omega^2=\tau\wedge\tau + dr\wedge\sigma\wedge dr\wedge\sigma +
r^2d\sigma\wedge d\sigma + 2\tau\wedge dr\wedge \sigma +2r\tau\wedge d\sigma + 2rdr\wedge\sigma\wedge d\sigma\,.
$$
Therefore,
$$
\omega^2 = 2r(\tau + dr\wedge\sigma)\wedge d\sigma\,,
$$
and, consequently,
$$
\omega^2 \neq 0\,,
$$
where $r \neq 0$. We have
\begin{eqnarray*}
\omega^3 &=& 2r\left(\tau^2\wedge d\sigma + dr\wedge\sigma\wedge\tau\wedge d\sigma +
\tau\wedge d\sigma\wedge dr\wedge \sigma + dr\wedge \sigma\wedge d\sigma \wedge dr\wedge \sigma \right)+\\
& {}& +\, 2r^2\left( \tau\wedge (d\sigma)^2 +
 dr\wedge \sigma\wedge (d\sigma)^2 \right)=0\,.
\end{eqnarray*}
Therefore, at each point, $\omega$ has rank $4$ or $\Omega =0$.
\end{proof}

The condition in Proposition \ref{non-pure} that $W$ has a $1$-form with certain properties is often satisfied. Indeed, let $\T^6=\R^6\backslash\Z^6$ be the $6$-dimensional standard torus. Denote by $(x_1,\ldots ,x_6)$ coordinates on $\R^6$. Let $W=\{[(x^1,\ldots ,x^6)]\in\T^6\,\,\,\vert\,\,\, x_5=x_6=0\}$. Then the Poincar\'e dual $\mathfrak{a}$ of $W$ is given by $dx_5\wedge dx_6$. Set
$$
\sigma =\cos(2\pi x_3)dx_1 + \sin(2\pi x_3)dx_2\,.
$$
Then
$$
d\sigma = 2\pi\left(\sin(2\pi x_3) dx_1\wedge dx_3 -\cos(2\pi x_3) dx_2\wedge dx_3 \right)
$$
is never vanishing on $W$. We immediately obtain
$$
d(\sigma\wedge d\sigma) =0\,.
$$

\end{document}